\documentclass[11pt]{amsart}
\usepackage{eucal}
\usepackage{amssymb}
\usepackage{epsfig}
\usepackage{amscd}
\usepackage{enumerate}

\newcommand{\ra}{{\rightarrow}}

\renewcommand{\H}{\mathbb{H}}

\newcommand{\R}{\mathbb{R}}
\newcommand{\C}{\mathbb{C}}

\newcommand{\Fr}{\operatorname{Fr}}

\DeclareMathOperator{\Int}{Int}

\def\ie{i.e. }
\def\pslc {\mathrm{PSL}(2,\C)}

\theoremstyle{plain}
\newtheorem{thm}{Theorem}[section]

\newtheorem{lem}[thm]{Lemma}
\newtheorem{prop}[thm]{Proposition}

\newtheorem*{thm:main}{Theorem \ref{thm:main}}

\numberwithin{equation}{section}

\begin{document}

\title{Measurable rigidity for Kleinian groups}
\author{Woojin Jeon and Ken'ichi Ohshika}
\date{\today}
\maketitle
\begin{abstract}
Let $G, H$ be two Kleinian groups with homeomorphic quotients $\H^3/G$ and $\H^3/H$. We assume that $G$ is of divergence type, and consider the Patterson-Sullivan measures of $G$ and $H$.
The measurable rigidity theorem by Sullivan and Tukia says that a measurable and essentially directly measurable equivariant boundary map $\widehat k$ from the limit set $\Lambda_G$ of $G$ to that of $H$ is either the restriction of a M\"{o}bius transformation or totally singular.
In this paper, we shall show that such $\widehat k$ always exists. In fact, we shall construct $\widehat k$ concretely from the Cannon-Thurston maps of $G$ and $H$. 
\end{abstract}

\section{introduction}

In the theory of Kleinian groups, it is very important to consider deformations and rigidity.
Quasi-conformal deformations of Kleinian groups have been studied for a long time, starting from the work of Ahlfors, Bers, Kra, Maskit and Marden.
These deformations are given by considering Beltrami differentials supported on the domains of discontinuity of Kleinian groups.
On the other hand, Sullivan showed in \cite{Sul81} that there are no non-trivial quasi-conformal deformations supported on limit sets, which is derived from the ergodic property of the actions of Kleinian groups on the product of two Riemann spheres.
In particular, when the limit set is the entire Riemann sphere, his theorem implies that there is no non-trivial quasi-conformal deformation.

Sullivan moreover showed what can be called the measurable rigidity in \cite{Sul82} as follows.
Consider two isomorphic Kleinian groups $G$ and $H$ assuming that $G$ is of divergence type, and suppose that
their limit sets $\Lambda_G$ and $\Lambda_H$  have conformal measures $\mu_G$ and $\mu_H$ of the same dimension.
Suppose furthermore that there is a Borel isomorphism $f$ from the Riemann sphere to the Riemann sphere conjugating $G$ to $H$ almost everywhere. 
The measurable rigidity says that if $f^*\mu_H$ and $\mu_G$ are absolutely continuous to each other,
then $f$ almost everywhere coincides with a conformal automorphism conjugating $G$ to $H$.

Tukia showed a similar rigidity in more general settings in \cite{Tuk}.
We present his result restricting it to the case with which our theorem is concerned.
Let $G$ and $H$ be Kleinian groups as in the previous paragraph, and $\mu_G, \mu_H$ conformal measures where $\mu_G$ does not have an atom.
It is not assumed that the dimensions of $\mu_G$ and $\mu_H$ are the same, but is assumed that at least one of their dimensions is positive.
Suppose that there is an essential injective (\ie injective outside a $\mu_G$-null set), measurable, and essentially directly measurable (\ie the image of any measurable set outside some fixed $\mu_G$-null set is measurable) map $f: \Lambda_G \rightarrow \Lambda_H$ which conjugate $G$ to $H$ almost everywhere.
Then either there is a set $A$ in $\Lambda_G$ of full measure with $\mu_H(f(A))=0$ or $f$ coincides with the restriction of a conformal automorphism to $\Lambda_G$ almost everywhere and the dimensions of $\mu_G$ and $\mu_H$ coincide.

In the case when $H$ is a quasi-conformal deformation of $G$, it is obvious that there is a Borel map conjugating $G$ to $H$ as in the results above of Sullivan and Tukia.
But the existence of such a map for general isomorphic Kleinian groups $G$ and $H$, in particular if we require it to be essentially directly measurable as in Tukia's setting, is not so evident.
There is work in more general framework of  \cite[Proposition 4.3.9]{Z}, \cite[Theorem 3.1]{SZ} and \cite[Theorem 0.2]{BM}, which implies the existence of a Borel map conjugating $G$ to $H$, but still we do not know such a map is essentially directly measurable.

On the other hand,  recent developments of Kleinian group theory have made it possible to construct an equivariant map from the limit set of $G$ to that of $H$  concretely using Cannon-Thurston maps.
In this paper, we shall show such a map constructed from Cannon-Thurston maps is in fact satisfies the conditions of  Tukia's theorem provided that $G$ is of divergence type.
To be more precise, we shall prove the following.

\begin{thm}\label{thm:main}
Let $F$ be a geometrically finite, minimally parabolic, non-elementary, torsion-free
Kleinian group.
Suppose that we are given two representations $\rho_1, \rho_2 \in \mathcal{D}(F)$ with images $G$ and $H$ in $\pslc$ such that
there are homeomorphisms from $\H^3/F$ to $\H^3/G$ and $\H^3/H$ which induce $\rho_1$ and $\rho_2$  respectively between the fundamental groups.
We further assume that $G$ is of divergence type.
Let 
$\widehat i :\Lambda_F \ra \Lambda_G$ and $\widehat j :\Lambda_F \ra \Lambda_H$ be the Cannon-Thurston maps, which are guaranteed to exist by virtue of the work of Mj.
Let $\mu_F, \mu_G$ and  $\mu_H$ be the Patterson-Sullivan measures of $F, G$ and $H$ with the base point at the origin of the Poincar\'{e} ball. 
Then the following hold.
\begin{enumerate}
\item There exists a Borel  set $Z \subset \Lambda_G$  with $\mu_G(Z)=0$ such that $\widehat j \widehat i^{-1}$
is homeomorphism from $\Lambda_G \setminus Z$ to its image.  
\item Either $\widehat j \widehat i^{-1}$ is the restriction of a M\"{o}bius transformation or there exists 
a Borel  subset $A$ of $\Lambda_G\setminus Z$ such that $\mu_G(A)=\mu_G(\Lambda_G)$ and $\mu_H(\widehat j\widehat i^{-1}(A))=0$.  
\end{enumerate}
\end{thm}

The Cannon-Thurston map $\widehat i$ of a Kleinian group $G$ is defined to be a continuous equivariant map from the limit set  $\Lambda_F$ of $F$ to the limit set $\Lambda_G$ of $G$, where $F$ is a minimally parabolic geometrically finite group isomorphic to $G$ (see section \ref{DF}). 
Although $\widehat i^{-1}$ may not be well defined as a map, we can show that $\widehat i$ is a Borel measurable isomorphism under the condition that $G$ is of divergence type. Then we can define $\widehat k$ to be $\widehat j \widehat i^{-1}$ and can show that it is both Borel measurable and Borel essentially directly measurable.

\section{Preliminaries}
\subsection{Hyperbolic 3-manifolds}\label{hyp}
We use the symbol $\H^3$ to denote the hyperbolic 3-space. 
We need to use both the Poincar\'{e} ball model and the upper half space model of $\H^3$ depending on the situation.
The ball model is the open unit ball $\{x\in \R^3 \mid \vert x\vert <1\}$ with the metric $2\vert dx \vert/\sqrt{1-\vert x\vert^2}$ where $\vert x \vert$ is the usual Euclidean norm of $x$, and the upper half space model is $\{x=(x_1, x_2, x_3) \in \R^3 \mid x_3>0\}$ with the metric $\vert dx\vert/x_3$.
The ideal boundary of the upper half space can be identified with $\hat\C=\C \cup\{\infty\}$,
and is also denoted by $S^2_\infty$.
A discrete subgroup $G$ of $\pslc$ is called a {\it Kleinian group}. If $G$ has no elliptic elements, the quotient $N:=\H^3/G$ becomes a hyperbolic 3-manifold. 
Throughout this paper, every Kleinian group is assumed to be   finitely generated and without elliptic elements.

For a Kleinian group $G$, its {\it limit set} $\Lambda_G$ is the closure of the set of accumulation points on $S^2_\infty$ of an orbit $G \cdot o$ with $o\in\H^3$. 
The complement $\Omega_G:=S^2_\infty \backslash \Lambda_G$ is the domain of discontinuity for the action of $G$ on $\hat\C$. 
The {\it convex hull} of $G$ is defined to be the smallest convex set in $\H^3$ whose closure in $\H^3\cup S^2_\infty$ contains $\Lambda_G$. The {\it convex core} of $G$ is the quotient by $G$ of its convex hull.
The Kleinian group $G$ is called {\it geometrically finite} (resp.\ {\em convex cocompact}) if its convex core has finite volume (resp.\ if its convex core is compact). 
The Kleinian group $G$ is called {\it minimally parabolic} if every parabolic element  is contained in a rank-2 free abelian subgroup of $G$.

The thin part $N_\epsilon\subset N$ is the set of points in $N$ where the injectivity radius is less than $\epsilon$. It is known that there exists some $\epsilon>0$, called Margulis constant, such that 
each component of $N_\epsilon$ is either a {\it cusp neighbourhood} or a solid torus neighbourhood of a closed geodesic, the latter of which is called a {\it Margulis tube}.
Let $N_0$ be the {\it noncuspidal} part of $N$, which is the complement of the cusp neighbourhoods in $N$ for some Margulis constant $\epsilon$. 
Then we can find a {\it relative compact core} (see \cite{FM, Sco}) $C\subset N_0$, which  is a compact 3-manifold with boundary such that the inclusion of $(C, C \cap \partial N_0)$ into $(N_0, \partial N_0)$ is a homotopy equivalence of pairs. 
The ends of $N_0$ are in one-to-one correspondence with the components of $\Fr C$ in $N_0$. 
When an end $e$ corresponds to a component $S_e$ of $\Fr C$, we say that $e$ {\it faces} $S_e$. 

An end of $N_0$ is called {\it geometrically finite} if it has a neighbourhood which is disjoint from the convex core of $G$.
Clearly $G$ is geometrically finite if and only if every end of $\H^3/G$ is geometrically finite.
A geometrically finite end $e$ corresponds to a component $\Sigma_e$ of $\Omega_G/G$ attached to $e$ in the Kleinian manifold $N_G:=(\H^3\cup\Omega_G)/G$. 
The point of the Teichm\"{u}ller space of $\Sigma_e$ determined by the conformal structure coming from $\Omega_G$ is called the {\it conformal structure at infinity} of $e$. 

%
\subsection{Deformation spaces}\label{DF}
For a group $F$ we denote by $\mathcal D(F)$ the deformation space of $F$, that is, the space of faithful discrete representations of $F$ into $\pslc$ modulo conjugacy, endowed with the topology induced from the representation space.
To denote a point in $\mathcal D(F)$, we use either a representation $\phi: F \rightarrow \pslc$ in the class or the Kleinian group $\phi(F)$ or when necessary, the pair $(F, \phi)$.
By Thurston's uniformisation theorem, every finitely generated Kleinian group can be regarded as a point in $\mathcal D(F)$ for a minimally parabolic geometrically finite Kleinian group $F$.

\subsection{Ending laminations}\label{ending}
For a closed hyperbolic surface $S$, 
a {\it geodesic lamination} $\lambda$ on $S$ is a closed set consisting of disjoint simple geodesics called {\it leaves}.
A leaf of $\lambda$ is said to be isolated if no point of the leaf is an accumulation points of points on the other leaves.
A minimal component of $\lambda$ is a sublamination $\mu$ of $\lambda$ such that the closure of each leaf of $\mu$ coincides with $\mu$.
A geodesic lamination is decomposed into finitely many minimal components and finitely many isolated leaves which are not closed leaves.

Each component $V$ of $S\setminus\lambda$ is called as a {\it complementary region} of $\lambda$ in $S$.
A lift of a complementary region to $\H^2$ is an ideal polygon with countably many ideal vertices.
If a complementary region itself is simply connected, it is a ideal polygon with finite vertices, which coincides with its lift in $\H^2$.
We say that a geodesic lamination is {\it filling} when each of its complementary regions is simply connected.

Now we shall define the ending lamination $\lambda_e$ for a geometrically infinite end $e$ of $N=\H^3/G$ as follows.
Let $N$ be a hyperbolic 3-manifold and $e$ an end of $N_0$.
As explained before, $e$ faces a component $S_e$ of the frontier of the relative compact core $C$.
Suppose that $e$ is geometrically infinite.
By the tameness theorem proved by Bonahon \cite{Bona}, Agol \cite{Ag} and Calegari-Gabai \cite{CG}, we see that we can take $C$ so that the complement of $C$ containing $e$, which we denote by $U_e$, is homeomorphic to $S_e \times (0,\infty)$.
Since $e$ is geometrically infinite, every neighbourhood of $e$ intersects a closed geodesic, and moreover, as was shown in Bonahon \cite{Bona} and Canary \cite{Ca}, there is a sequence of simple closed curves $\gamma_n$ in $S_e$ which are homotopic in $U_e \cup S_e$ to closed geodesics  tending to $e$.
Fix a hyperbolic metric on $S_e$ and regard $\gamma_n$ as a closed geodesic on $S_e$.
We consider the Hausdorff limit of $\gamma_n$.
It can be shown that this limit geodesic lamination has only one minimal component $\lambda_e$, which is filling.
We call $\lambda_e$ the {\em ending lamination} of $e$, which can be proved to be unique by Bonahon's intersection lemma appearing in \cite{Bona}.
The ending lamination theorem proved by Brock, Canary and Minsky in \cite{MinE, BCM1} says that the end invariant consisting of conformal structures at infinity and of ending laminations uniquely determines the Kleinian group $G$ up to conjugacy.

For our later use, we need to consider the union of ending laminations together with parabolic locus.
Consider a Kleinian group $(G,\phi)$ contained in $\mathcal D(F)$ for a minimally parabolic geometrically finite Kleinian group $F$, and suppose that there is a homeomorphism $\Phi : \H^3/F \rightarrow \H^3/G$ inducing an isomorphism conjugate to $\phi$ between the fundamental groups.
Let $C$ be a convex core of $\H^3/F$.
We can isotope $\Phi$ so that $\Phi(C) \cap (\H^3/G)_0$ is a relative compact core of $(\H^3/G)_0$ such that each component $U$ of $(\H^3/G)_0 \setminus \Phi(C)$ is homeomorphic to
$S \times (0,\infty)$ where $S$ is the frontier component of $\Phi(C)$ touching $U$.
We fix a hyperbolic metric on $\partial C$.
Then we consider the union $\lambda$ of all ending laminations and core curves of the annuli among $\Phi(C) \cap \Fr (\H^3/G)_0$ and pull them back to $\partial C$ by $\Phi^{-1}$.
We call the geodesic lamination on $\partial C$ obtained by  isotoping $\lambda$ the {\em total ending lamination} for $(G,\phi)$.

\subsection{Cannon-Thurston maps}\label{CT}
Let $(G, \rho)$ be a finitely generated Kleinian group in $\mathcal D(F)$ for a minimally parabolic geometrically finite Kleinian group $F$.
Fixing a point $o \in \H^3$, we can construct a $\rho$-equivariant map $i$ between the orbit  $F\cdot o$ and the orbit of $G\cdot o$ by setting $i(o)=o$.
If $i$ extends to a continuous equivariant map between the limit sets $\widehat i : \Lambda_F \rightarrow \Lambda_G$, the latter map is called the Cannon-Thurston map.
The existence of Cannon-Thurston maps was proved  by Mj for  Kleinian surface groups in \cite{MjSurf}, and the general case was dealt with in \cite{Mj}.

Mj also gave a criterion for limit points of $F$ to be non-injective under $\widehat i$ as follows.
Consider the Kleinian manifold $(\H^3\cup \Omega_F)/F$.
Then by the nearest point map, we can identify $\partial C$ with $\Omega_F/F$.
Let $\lambda$ be the total ending lamination for $(G,\phi)$ regarded as lying on $\Omega_F/F$.
We consider the lift $\widetilde \lambda$ of $\lambda$ to the universal cover $\Omega_F$.
We define an equivalence relation $R$ on the points on $\Lambda_F$ to be the smallest one having the property that  $xRy$ if $x$ and $y$ are the two endpoints of a leaf of $\widetilde \lambda$.
What Mj showed is that $\widehat i(x)=\widehat i(y)$ if and only if $xRy$.

\subsection{Conformal density}\label{conf}
For $x$ in $\H^3$ and $\eta \in S^2_\infty$, we take a geodesic ray $\alpha(t)$ from $\alpha(0)=x\in \H^3$ toward $\eta\in S^2_\infty$.
Then we define the {\it Busemann function} $b_{x, \eta}(\cdot)$ at $\eta$ with $b_{x, \eta}(x)=0$ by $$b_{x, \eta}(y)=\lim\limits_{t\ra \infty}(d_{\H^3}(y, \alpha(t))-d_{\H^3}(x, \alpha(t))).$$
For $x,y$ in $\H^3$ and $s \geq 0$, we consider the Poincar\'{e} series for a Kleinian group $G$, defined by $$g_s(x,y)=\sum_{g \in G} e^{-sd_{\H^3}(x,g\cdot y)}.$$
The critical exponent $\delta_G$ of $G$ is defined to be
$$\delta_G=\limsup\limits_{r \ra \infty}\frac{1}{r}\log(\#\{g\in G \mid d_\H^3(o, g\cdot o)\leq r\}).$$
We say that $G$ is of  {\it divergence type} if $g_s(x,y)$ diverges at $s=\delta_G$. 
Otherwise $G$ is said to be of {\it convergence type}.
These definitions are independent of the choice of $x,y\in \H^3$.

For $x \in \H^3$ and $s \geq 0$, we consider a measure defined by $$\mu_x^s=\frac{1}{g_s(y,y)}\sum_{g \in G} e^{-sd_{\H^3}(x,g\cdot y)}\delta_{g\cdot y},$$ where $\delta_{g\cdot y}$ denotes the Dirac measure of weight $1$ at $g\cdot y$. 
Then we can find a sequence $\{s_i\}$ with $s_i \ra \delta_G+0$ such that $\mu_x^{s_i}$ weakly converges to a finite measure $\mu_x$ on the compact space $\H^3\cup S^2_\infty$.
When $G$ is of divergence type, $\mu_x$ has its support on the limit set $\Lambda_G$. 
The family of measures $\{\mu_x\}$ has two properties $$\frac{d\mu_x}{d\mu_y}(\eta)=e^{-\delta_G b_{x,\eta}(y)},\   g^{*}\mu_x=\mu_{g^{-1}\cdot x},$$
where the pull-back measure $g^*\mu_x$ is defined by setting $(g^*\mu_x)(A):=\mu_x(g\cdot A)$ for any $\mu_x$-measurable set $A\subset \Lambda_G$.
A family of measures parameterised by $x \in \H^3$ having these properties is called a $G$-invariant {\it conformal density} of dimension $\delta_G$.

Even when $G$ is of convergence type, it is still possible to construct a conformal density $\{\mu_x\}$ of dimension $\delta_G$ by increasing the Dirac mass on each orbit point suitably as was shown in \cite{Pa, Sul79}.
We call an element of a $G$-invariant conformal density a {\it conformal measure}. 
In particular, the conformal measure $\mu_x$ constructed above is called as the {\it Patterson-Sullivan measure}  of $G$ based at $x$ whether or not $G$ is of divergence type.

The dimension $\delta_G$ of a conformal density $\{\mu_x\}$ is always positive if $G$ is {\it non-elementary} i.e., if $\Lambda_G$ has more than  two points (see \cite[Corollary 2]{Sul79} or \cite[Proposition 3.8]{CS}).
A trivial example of a 2-dimensional conformal density on $S^2_\infty$ is the area density $\{area_x\}$ obtained by identifying the visual sphere based at the point $x\in \H^3$ with $S^2_\infty$.
If we use the Poincar\'{e} ball model of $\H^3$ setting $o$ to be the origin of the ball, then $area_o$ is equal to the Lebesgue measure on $S^2_\infty$ up to constant multiple.
As an application of the tameness theorem combined with the work of Culler and Shalen in \cite[Proposition 3.9.]{CS}, it is known that if a Kleinian group $G$ satisfies $\Lambda_G=S^2_\infty$ then any $G$-invariant conformal measure $\mu_x$ is equal to the area density up to homothety. In fact, we have $K\mu_x=area_x$ with a constant $K >0$ for any $x\in \H^3$.

A point in $\Lambda_G$ is called {\it a conical limit point} if it is an accumulation point of an orbit $G\cdot o$ in $\H^3$ contained in a fixed radial neighbourhood of a geodesic ray ending at that point. We denote the set of conical limit points by $\Lambda_{H,c}$. 
We use the symbol $\Delta$ to denote the diagonal set in general.
For a subset $X$ of $S^2_\infty$, we use the symbol $X^{(2)}$ to denote $(X \times X) \setminus \Delta$.

\begin{thm}(Hopf, Tsuji, Sullivan, Thurston, Canary and Yue)\label{conicalfull} 
Let $G$ be a finitely generated torsion-free Kleinian group, and $\mu_x$ its Patterson-Sullivan measure.
Then the following are equivalent.
\begin{enumerate}
\item The conical limit set $\Lambda_{G,c}$ has  full $\mu_x$-measure.
\item The conical limit set $\Lambda_{G,c}$ has positive $\mu_x$-measure.
\item The diagonal action of $G$ on $(S^2_\infty)^{(2)}$ is $\mu_x\times \mu_x$-ergodic.
\item $G$ is of divergence type.
\item $G$ is either geometrically finite or $\Lambda_G=S^2_\infty$.
\end{enumerate}

\end{thm}
\begin{proof}
The fourth and the fifth are equivalent as was shown in  Section 9 of the paper \cite{Ca} by Canary.
By Sullivan\cite[pp. 3]{Sul79}, we see that the first three conditions are equivalent and imply the fourth, and that the fourth implies the first three if $\delta_G>1$.
Suppose that $G$ is of divergence type, then we have the following.
If $G$ is not geometrically finite, then $\delta_G=2$ because $\Lambda_G=S^2_\infty$ by the fifth condition and $K\mu_o=area_o$ for some constant $K>0$. 
Therefore as explained above this implies the first three.
If $H$ is geometrically finite, then $\Lambda_G$ is the union of $\Lambda_{G,c}$ and parabolic fixed points, which are countable, and hence $\Lambda_{G,c}$ has the full $\mu_x$-measure.
Thus the fourth implies the first three in both of the two cases. 
The equivalence of the first four conditions also follows from a more general result by Yue in \cite[pp. 77]{Yue2}.
\end{proof}

\subsection{Measure theoretic notions}
Let $X$ be a complete separable metric space with the Borel $\sigma$-algebra generated by the open sets. A measure space $(Y, \nu)$ is called {\it standard} or {\it Lebesgue} if it is measurably isomorphic to a Borel subset of $X$. 
Two standard measure spaces are Borel isomorphic if and only if they have the same cardinality.

For a discrete group $G$ with the counting measure, a {\it measurable action} of $G$ on a standard measure space
$(Y, \nu)$ is defined to be a measurable map $\Theta:G\times Y \ra Y$ from the product measure space satisfying the following properties.

\begin{enumerate}
\item $\Theta(id, y)=y$ for all $y\in Y$.
\item $\Theta(g_1, \Theta(g_2, y))=\Theta(g_1g_2, y)$.
\end{enumerate}

Since $G$ is a group, $\Theta(g, \cdot)$ has the measurable inverse $\Theta(g^{-1},\cdot)$ for any $g\in G$.
We denote $\Theta(g, y)$ simply by $g\cdot y$.
The measure $\nu$ is said to be {\it $G$-invariant} if $\nu(A)=\nu(g\cdot A)$ for any measurable set $A$ in $Y$ and $g\in G$. 
For an element $g\in G$, we define the push-forward $g_*\nu$ and the pull-back $g^*\nu$  by setting $g_*\nu(A)=\nu(g^{-1}\cdot A)$ and $g^*\nu(A)=\nu(g\cdot A)$. 
We say that $\nu$ is {\it quasi-invariant} under $G$ if $g_*\nu$ and $\nu$ are equivalent for all $g \in G$, i.e., if $\nu(A)=0 \Leftrightarrow g_*\nu(A)=0$.

A measurable {\it $G$-space} is a standard measure space $(Y, \nu)$ with a measurable $G$-action such that $\nu$ is quasi-invariant under $G$. 
We say that  a set $Z\subset Y$ is {\it null} in $(Y,\nu)$ if $Z$ is a subset of a measurable set $Z'$ with $\nu(Z')=0$. 
A set is said to be  {\it conull} when it is  the complement of a null set.    
A measurable $G$-space is called {\it ergodic} if any $G$-invariant measurable set is either null or conull. 
This definition is equivalent to the following weaker condition (see \cite[Chapter 2]{FK}, for example):
a measurable $G$-space is ergodic if and only if any essentially invariant measurable set is either null or conull.
Another way to express this condition is:  a $G$-space is ergodic if and only if any essentially $G$-invariant measurable function from the $G$-space is constant almost everywhere.

For a Kleinian group $G$ with a conformal measure $\mu_x$ on $S^2_\infty$, both 
$(S^2_\infty, \mu_x)$ and $((S^2_\infty)^{(2)}, \mu_x\times \mu_x)$
are $G$-spaces. 
If $G$ is of divergence type, both of them are ergodic $G$-spaces. 
If $G$ is geometrically infinite and of divergence type, as mentioned in Section \ref{conf}, we know that a conformal density is equal to the area density up to constant multiple.
If $G$ is geometrically finite then the dimension of a conformal density is equal to the Hausdorff dimension of $\Lambda_G$. Furthermore, any two ergodic conformal densities of the same dimension are equal up to a constant multiple (see \cite[pp. 181]{Sul79}). 
Thus, the conformal density is unique (up to constant multiples) if $G$ is of divergence type.
On the other hand, if $G$ is of convergence type, conformal densities are neither unique nor ergodic in general. 
For this, see \cite{AFT, BJ}.

\section{Non-injective points}
Let $F$ be a minimally parabolic geometrically finite Kleinian group and $(G,\phi)$ a Kleinian group in $\mathcal D(F)$ such that there is a homeomorphism from $\H^3/F$ to $\H^3/G$ inducing an isomorphism conjugate to $\phi$ between the fundamental groups.
We let $\widehat i :\Lambda_F \ra \Lambda_G$ be the Cannon-Thurston map, which is guaranteed to exist by the work of Mj.
We shall need the following proposition to show that $\widehat i$ is a Borel measurable isomorphism.
\begin{prop}\label{Complementary}
There are only countably many points  $\eta\in \Lambda_F$ such that the cardinality of $\widehat i^{-1}(\widehat i(\eta))$ is greater than $2$.
\end{prop}

Let $N_F$ be $(\H^3 \cup \Omega_F)/F$, which is either a compact 3-manifold or a 3-manifold with only finitely many ends, all of which are toral.
Let $\Phi_1: \Int N_F \rightarrow \H^3/G$ be a homeomorphism which was assumed to exist.
Let $S$ be a component of $\Omega_F/F$.
Let $\lambda_G$ be the total ending lamination for $(G,\phi)$ as explained in Section \ref{ending}, and consider its restriction $\lambda_G \cap S$ to $S$, which we call the total ending lamination on $S$. 

Now, let $\Omega_0$ be a component of $\Omega_F$ which covers $S$ by the canonical projection from $\Omega_F$ to $\Omega_F/F$, and  $p_0 : \Omega_F \rightarrow S$ the restriction of the projection.
Each leaf of the total ending lamination on $S$ lifts to open arcs which are geodesics with respect to a hyperbolic metric on $\Omega_0$ inducing one on $S$, which we call {\it end invariant leaves} for $S$.
More specifically, we call open arcs which are lifts of leaves of ending laminations as {\it ending lamination leaves} and those which are lifts of parabolic curves as {\it parabolic leaves}. 
A point in $\Lambda_G$ is said to be an {\it end invariant point} for $S$ when it is an endpoint at infinity of a lift of an end invariant leaf of $S$.

To prove the proposition, we shall first show the following lemma.

\begin{lem}
\label{three points}
Suppose that there are three distinct points $x, y, z \in \Lambda_F$ with $\widehat i(x)= \widehat i(y)= \widehat i(z)$.
Then there is a component $S$ of $\Omega_F/F$ and the total ending lamination $\lambda:=\lambda_S(G)$ of $G$ on $S$ as follows.
Either there is a lift $U$ in $\Omega_F$ of a complementary region of $\lambda$ whose vertices at infinity contain all of $x, y,z$ or there exist two adjacent lifts $U_1, U_2$ in $\Omega_F$ of complementary regions of $\lambda$ such that $\overline U_1 \cap \overline U_2$ contains a parabolic leaf and all of $x,y,z$ are contained in the vertices of $U_1 \cup U_2$. 
\end{lem}
\begin{proof}
By Mj's work explained in Section \ref{CT}, the relation between two points $x, y$ in $\Lambda_F$ defined by $\widehat i(x)=\widehat i(y)$ is the transitive closure of the relation that two points are the endpoints of the same end invariant leaf.
This implies that if we only take into account one component $S$ of $\Omega_F/F$,  three points $x, y,z$ which are equivalent  must be situated as in the statement.
Therefore, to show our lemma, it is enough to prove that  for two distinct components $S_1$ and $S_2$ of $\Omega_F/F$, an end invariant point for $S_1$ cannot be that of $S_2$ at the same time.

Now, seeking a contradiction, suppose that there is a point $z \in \Lambda_F$ which is an end invariant point  for both $S_1$ and $S_2$.
Let $\Omega_1$ and $\Omega_2$ be the components of $\Omega_F$ covering $S_1$ and $S_2$ respectively.
Then there exist geodesic rays $r_1$ and $r_2$ ending at $z$ which are lifts of parabolic curves or leaves of ending laminations to $\Omega_1$ and $\Omega_2$ respectively.
Let $l_1$ and $l_2$ be their projections on $S_1$ and $S_2$.
Since the total ending lamination is decomposed into minimal components, the closures of $l_1$ and $l_2$ are minimal components of the total ending lamination, which we denote by $\lambda_1$ and $\lambda_2$.
Now, the argument of \cite[Lemma C.1]{Le} (see also \cite{LeT}) implies that $\lambda_1$ and $\lambda_2$ are homotopic in $N_F$.
Here two laminations $\lambda_1$ and $\lambda_2$ are said to be homotopic if there is a sequence of incompressible annuli $A_i$ properly embedded in $N_F$ such that $A_i \cap \partial N_F$ converges to $\lambda_1 \cup \lambda_2$ in the Hausdorff topology.
Since $\lambda_1$ and $\lambda_2$ are parabolic curves or ending laminations of $G$ lying on distinct components of $\Omega_F/F$, this is a contradiction.
\end{proof}
\begin{proof}[Proof of Proposition \ref{Complementary}]
Suppose that $z \in \Lambda_F$ is a point such that $\widehat i^{-1}(\widehat i(z))$ has more than two points.
Then by Lemma \ref{three points},  there is a component $S$ of $\Omega_F/F$, and $z$ is an ideal vertex of a lift of a complementary region of the total ending lamination for $G$ on $S$, which we denote by $\lambda$.
Note that there are only finitely many complementary regions of $\lambda$, and each of them has only finitely many ideal vertices.
Therefore, there are only countably many ideal vertices for the lifts of complementary regions.
This implies that there are only countably many points $z$ with $\#( \widehat i^{-1}(\widehat i(z))) >2$.
\end{proof}

\section{measurable rigidity}
In this section, we shall prove our main theorem.
\begin{thm:main}
Let $F$ be a geometrically finite, minimally parabolic, non-elementary, torsion-free
Kleinian group.
Suppose that we are given two representations $\rho_1, \rho_2 \in \mathcal{D}(F)$ with images $G$ and $H$ in $\pslc$ such that
there are homeomorphisms from $\H^3/F$ to $\H^3/G$ and $\H^3/H$ which induce $\rho_1$ and $\rho_2$  respectively between the fundamental groups.
We further assume that $G$ is of divergence type.
Let 
$\widehat i :\Lambda_F \ra \Lambda_G$ and $\widehat j :\Lambda_F \ra \Lambda_H$ be the Cannon-Thurston maps, which are guaranteed to exist by virtue of the work of Mj.
Let $\mu_F, \mu_G$ and  $\mu_H$ be the Patterson-Sullivan measures of $F, G$ and $H$ with the base point at the origin of the Poincar\'{e} ball. 
Then the following hold.
\begin{enumerate}
\item There exists a Borel  set $Z \subset \Lambda_G$  with $\mu_G(Z)=0$ such that $\widehat j \widehat i^{-1}$
is homeomorphism from $\Lambda_G \setminus Z$ to its image.  
\item Either $\widehat j \widehat i^{-1}$ is the restriction of a M\"{o}bius transformation or there exists 
a Borel  subset $A$ of $\Lambda_G\setminus Z$ such that $\mu_G(A)=\mu_G(\Lambda_H)$ and $\mu_H(\widehat j\widehat i^{-1}(A))=0$.  
\end{enumerate}
\end{thm:main}

The following Lemma \ref{conical} is implied from a more general result \cite[Proposition 7.5.2]{Ger}.
See also \cite[Theorem A]{JKLO}.
\begin{lem}\label{conical}
If the cardinality of $\widehat i^{-1}(\xi)$ is greater than $1$, then $\xi\in \Lambda_G$ is not a conical limit point. 
\end{lem}

\begin{lem}\label{atomless}
For a  Kleinian group $G$ of divergence type, the Patterson-Sullivan measure $\mu_G$ has no atom.
Therefore, every countable set in $\Lambda_G$ has $\mu_G$-measure zero.
\end{lem}
\begin{proof}
A  Kleinian group  of divergence type either is geometrically finite or has the entire $S^2_\infty$ as its limit set as in Theorem \ref{conicalfull}.
If $\Lambda_G=S^2_\infty$, then $\mu_G$ coincides with the area on $S^2_\infty$ up to a constant multiple as was explained in Section \ref{conf}, and hence it has no atom.
If $G$ is geometrically finite, then $\mu_G$ has no atom either as shown in \cite[Theorem 3.5.10]{Nic} for example. 
\end{proof}

We shall construct a Borel null set $Z$ in $\Lambda_G$ such that $\widehat{j}\widehat{i}^{-1}$ is well defined and injective in $\Lambda_G\setminus Z$. 
Set $$End_F(G) := \{(\xi_1, \xi_2)\in \Lambda_F^{(2)}\mid \widehat i(\xi_1)=\widehat i(\xi_2)\}, $$
$$\text{and } End_F(H) := \{(\eta_1, \eta_2)\in \Lambda_F^{(2)}\mid \widehat j(\eta_1)=\widehat j(\eta_2)\}.$$

\begin{lem}\label{null}
$End_F(G)$ and $End_F(H)$ are closed $F$-invariant Borel null sets in the $F$-space $(\Lambda_F^{(2)}, \mu_F\times\mu_F)$.
\end{lem}
\begin{proof}
We have only to show this for $End_F(G)$, for the argument for $End_F(H)$ is the same.
Since $\widehat i$ is continuous, $End_F(G)$ is closed in $\Lambda_F^{(2)}$.
We also see that $End_F(G)$ is $F$-invariant because $\widehat i$ is an equivariant map.
Since $(\Lambda_F^{(2)}, \mu_F\times \mu_F)$ is an ergodic $F$-space, we only need to show that $End_F(G)$ cannot have full $\mu_F\times\mu_F$-measure. 
Since $End_F(G)$ is closed in $\Lambda_F^{(2)}$ and $\Lambda_F^{(2)}$ is open in $\Lambda_F \times \Lambda_F$, it is sufficient to show that  $End_F(G)$ is not the entire $\Lambda_F^{(2)}$.

Choose a loxodromic element $\gamma\in F$ such that $\rho_1(\gamma)$ is also loxodromic.
Such an element exists because there are only finitely many conjugacy classes of accidental parabolic elements.
Let $\xi_1,\xi_2\in \Lambda_F$ be the fixed points of $\gamma$.
Then  $\widehat i(\xi_1)$ and $\widehat i(\xi_2)$ are conical limits points since they are loxodromic fixed points.
It follows that $(\xi_1, \xi_2)$ does not belong to $End_F(G)$ by Lemma \ref{conical}.
Therefore, $End_F(G)$ cannot be the entire $\Lambda_F^{(2)}$.
\end{proof}

Now we consider the projection maps $proj_1$ and $proj_2$ from $S^2_\infty\times S^2_\infty$ to $S^2_\infty$ sending $(\xi_1,\xi_2)$ to $\xi_1$ and to $\xi_2$ respectively. 
We set 
$$end_F(G):=proj_1(End_F(G))=proj_2(End_F(G)).$$

\begin{lem}\label{endnull}
$end_F(G)$ is a Borel null set in $(\Lambda_F, \mu_F)$.
Moreover $\widehat i$ is a homeomorphism from $\Lambda_F\setminus end_F(G)$ 
to $\Lambda_G\setminus \widehat i(end_F(G))$.
\end{lem}
\begin{proof}
We shall first prove that $end_F(G)$ is a Borel set.
We consider a countable covering $\{A_i\times B_i\}_{i\in I}$ of $\Lambda_F^{(2)}$ by  products of disjoint closed subsets $A_i, B_i$ of $\Lambda_F$. 
For each $A\times B \in \{A_i\times B_i\}$, we consider the set
$end_{A, A\times B}:=proj_1(End_F(G)\cap A\times B),$ 
and   $end_{B, A\times B}:=proj_2(End_F(G)\cap A\times B)$. 
In this setting, we claim that $end_{A, A\times B}$ is a closed subset of $A$.
To show this, we take a sequence $\{a_j\}$ in $end_{A, A\times B}$ converging to $a\in A$.
Then there exists a sequence $\{b_j\}\subset end_{B, A\times B}$ such that $(a_j, b_j)\in End_F(G)$ for each $j$. Since $End_F(G)$ is closed and $A\times B$ is compact, $\{(a_j, b_j)\}$ converges to $(a, b) \in End_F(G)$ for some $b\in B$ after passing to a subsequence. 
This shows $a\in end_{A, A\times B}$ and we have thus shown that $end_{A,A \times B}$ is closed.
Now, since $end_F(G)$ is expressed as a countable union $\bigcup\limits_{i\in I} end_{A_i, A_i\times B_i}$ of closed subsets, we see that $end_F(G)$ is a Borel set.

Secondly, we shall prove that $end_F(G)$ is a null set by extending Soma's argument in  \cite{Soma}.
Since $(\Lambda_F, \mu_F)$ is an ergodic $F$-space and $end_F(G)$ is an invariant set, 
 $end_F(G)$ is either null or conull with respect to $\mu_F$.
Suppose that $end_F(G)$ is conull.
Under this assumption, in the following we shall construct an invariant real valued function $\Psi$ which is defined on a conull set in $\Lambda_F^{(2)}$ and is Borel measurable. Since $(\Lambda_F^{(2)}, \mu_F\times \mu_F)$ is an ergodic $F$-space, $\Psi$ has to be constant almost everywhere and we shall reach a contradiction.

For simplicity, we denote $end_{A, A\times B}$ and $end_{B, A\times B}$ by $\mathsf{A}$ and $\mathsf{B}$ respectively.
Each $a\in \mathsf{A}$ corresponds to some $b\in \mathsf{B}$ with $(a,b) \in End_F(G)$ which may not be unique.
Let $Z_2$ be the set of the points $a\in \Lambda_F$ such that $\widehat i^{-1}(\widehat i(a))$ has more than two points.
Then by Proposition \ref{Complementary}, $Z_2$ is a countable set, and hence has measure zero by Lemma \ref{atomless}.

Now we claim that $\widehat i\vert_{\mathsf{A}\setminus Z_2}$ is a homeomorphism to its image.
Since we know that $\widehat i\vert_{\mathsf{A}\setminus Z_2}$ is a continuous injective map, we only need to show that it is also a closed map to its image.
Set $\mathsf{A}':=\mathsf{A}\backslash Z_2$, and let $C'$ be a closed subset of $\mathsf{A}'$. 
Then $C'=C\cap \mathsf{A}'$ for some compact set $C$ in $\Lambda_F$. 
It is clear that $\widehat i(C')\subset \widehat i(C)\cap \widehat i(\mathsf{A}\backslash Z_2)$. 
On the other hand, since $\widehat i^{-1}(\widehat i(Z_2))\cap \mathsf{A}=Z_2\cap \mathsf{A}$, we have  $\widehat i(C')=\widehat i(C)\cap \widehat i(\mathsf{A}\setminus Z_2)$.
Since $\widehat i(C)$ is compact, this shows that $\widehat i(C')$ is a closed set in  $\widehat i(\mathsf{A}\setminus Z_2)$, and hence $\widehat i\vert_{\mathsf{A}\setminus Z_2}$ is a homeomorphism to its image.

In the same way,  $\widehat i\vert_{\mathsf{B}'}$ for $\mathsf{B}':=\mathsf{B}\backslash Z_2$ is a homeomorphism to its image. 
Therefore we have a homeomorphism $\psi:= \widehat i\vert_{\mathsf{B}'}^{-1} \widehat i\vert_{\mathsf{A}'}$ from $\mathsf{A}'$ to $\mathsf{B}'$. When $A\times B=A_j\times B_j$
for some index $j\in I$, we denote this homeomorphism by $\psi_j$.

Now, we define a Borel measurable function $\Psi: (end_F(H)\backslash Z_2)^{(2)} \ra \R$ as follows.
For a given $(a,b)\in (end_F(H)\setminus Z_2)^{(2)}$, we choose two sets $A_j\times B_j$, $A_k\times B_k$ in the covering so that $a\in \mathsf{A}_j$ and $b\in \mathsf{A}_k$. 
Then we define $\Psi(a,b)$ to be the absolute cross ratio 
$$|cr(a, b, \psi_k(b),\psi_j(a))|=\frac{\vert a-\psi_k(b)\vert\vert \psi_j(a)-b\vert}{\vert a-b\vert\vert \psi_k(b)-\psi_j(a)\vert}.$$ 
By the definition of $\psi_j$ and $\psi_k$, it is clear that $\Psi(a,b)$ is determined only by $a$ and 
$b$, independently of the choice of $j, k$.
We shall drop the subscript of $\psi_j(a)$ when we do not need to specify in which $A_j$ lies $a$.
Since each $\psi_j$ is a homeomorphism, $\Psi$ is Borel measurable.
Moreover $\Psi$ is $F$-invariant because each $\psi_j$ is equivariant and the absolute cross ratio is preserved by M\"{o}bius transformations.

We shall show that $\Psi$ is not constant almost everywhere with respect to $\mu_F \times \mu_F$.
We choose an element $A\times B$ in the covering $\{A_i\times B_i\}_{i\in I}$ of $\Lambda_F^{(2)}$ such that $\mathsf{A}'$ has a positive measure with respect to $\mu_F$. 
This is possible because we are assuming $end_F(G)$ is conull.
We claim that there exists a sequence 
$\{B_{\epsilon_n}\}$ of $\epsilon_n$-open balls with respect to the spherical metric on $S^2_\infty$
such that the following holds.
For each $n$, $B_{\epsilon_n}\cap \mathsf{A}'$ has a positive $\mu_F$-measure and $B_{\epsilon_{n+1}}\subset B_{\epsilon_n}$, and the sequence $\{\epsilon_n\}$ converges to $0$ as $n \ra \infty$.
The existence of such balls can be seen by considering a finite covering of the compact set $\mathsf{A}$ by $\epsilon_n$-balls.

Next, we shall show that as $\epsilon_n\ra 0$, the infimum of $\Psi\vert_{({B_{\epsilon_n}\cap \mathsf{A}'})^{(2)}}$ goes to infinity.
Suppose not. 
Then there exist two sequences $\{a_n\},\{a_n'\}$ with $(a_n, a_n')\in ({B_{\epsilon_n}\cap \mathsf{A}'})^{(2)}$ such that $\Psi(a_n, a_n')<K$ for some constant $K$.
Now, since $\psi(\mathsf{A'})$ is contained in $\mathsf{B}$, there is a positive constant $L$ such that $|a_n-\psi(a_n')| \geq L$ and $|a_n'-\psi(a_n)|\geq L$.
On the other hand,  since $\vert a_n-a_n'\vert \leq \epsilon_n \ra 0$, it can be easily checked that $\Psi(a_n, a_n') \ra \infty$ as $n\ra \infty$, which is a contradiction.
Thus, we have shown that $\inf\Psi\vert_{({B_{\epsilon_n}\cap \mathsf{A}'})^{(2)}}$ goes to $\infty$ as $n \rightarrow \infty$.

Suppose now, seeking a contradiction, that $\Psi$ coincides with a constant $K'$ almost everywhere with respect to $\mu_F\times \mu_F$.
Then we choose $n$ large enough so that the infimum of $\Psi\vert_{(B_{\epsilon_n}\cap \mathsf{A}')^{(2)}}$ is greater than $K'$. This is a contradiction because $B_{\epsilon_n}^{(2)}$ has a positive $\mu_F\times \mu_F$ measure and thus there must be a point $(a,b)\in B_{\epsilon_n}^{(2)}$ such that $\Psi(a,b)=K'$.
Thus we have shown that $\Psi$ cannot be constant almost everywhere.

Finally, we see that $\widehat i\vert_{\Lambda_F\setminus end_F(G)}$ is also a homeomorphism to its image $\Lambda_H\setminus \widehat i(end_F(G))$ by noting that $\widehat i^{-1}(\widehat i(end_F(G)))=end_F(G)$ and that $\widehat i|_{\Lambda_F \setminus end_F(G)}$ is a closed map. 
\end{proof}

We can define $end_F(H)$ in the same way so that $\widehat j\vert_{\Lambda_F\setminus end_F(H)}$ becomes a homeomorphism to its image.
Now we define a subset $Z$ of $\Lambda_G$ as follows:
$$Z:= \widehat i(end_F(G)\cup end_F(H)).$$
\begin{lem}
$Z$ is a null set in $(\Lambda_G, \mu_G)$
\end{lem}
\begin{proof}
By Lemma \ref{conical}, $\widehat i(end_F(G))$ consists of nonconical limit points. Since $G$ is of divergence type, it has measure zero by Theorem \ref{conicalfull}. 
To see $\widehat i(end_F(H))$ is also null, 
we first note that $\widehat j\widehat i^{-1}\vert_{\Lambda_G\setminus \widehat i(end_F(G))}$ is a well-defined equivariant continuous map. This follows from the fact that $\widehat i$ gives a homeomorphism between $\Lambda_F\backslash end_F(G)$ and $\Lambda_G\backslash \widehat i(end_F(G))$ as in Lemma \ref{endnull} and $\widehat j$ is continuous. 
If we set $$End_G(H):=\{(\xi_1, \xi_2)\in (\Lambda_G\setminus\widehat i(end_F(G)))^{(2)} \vert \widehat j\widehat i^{-1}(\xi_1)=\widehat j\widehat i^{-1}(\xi_2)\},$$
 then $End_G(H)$ is a $G$-invariant closed subset in $(\Lambda_G\setminus\widehat i(end_F(G)))^{(2)}$.
We set $end_G(H):=proj_1(End_G(H))$, and can use the same argument as in the proof of Lemma \ref{endnull} to prove it is a Borel null set.

In fact, we can prove this as follows.
We consider a countable covering $\{C_j\times D_j\}_{j\in J}$ of $(end_G(H))^{(2)}$ where $C_j$ and $D_j$ are disjoint closed subsets of $\Lambda_G$ for each $j\in J$. 
Set $\mathsf{C}_j:=end_{{C_j},C_j\times D_j}=proj_1(end_G(H)\cap C_j\times D_j)$ which is closed in $\Lambda_G\backslash \widehat i(end_F(G))$.
We note that $end_G(H)=\bigcup\limits_{j\in J}\mathsf{C}_j$ to see that it is a Borel set in $\Lambda_G\backslash \widehat i(end_F(G))$. 
Then we apply Proposition \ref{Complementary} to $\widehat j$ to see $\{x\in \Lambda_F \mid \#(\widehat j^{-1}(\widehat i(x)))>2\}$ is countable, and hence its image $Z_2'$ by $\widehat i$ is also countable. 
We set $\mathsf{C}_j':=\mathsf{C}_j\backslash Z_2'$ and $\mathsf{D}_j':=\mathsf{D}_j\backslash Z_2'$ as before. Then
we also have a homeomorphism $\psi_j':=(\widehat j\widehat i^{-1}\vert_{\mathsf{C}_j'})^{-1}(\widehat j\widehat i^{-1}\vert_{\mathsf{D}_j'})$ between $\mathsf{C}_j'$ and $\mathsf{D}_j'$. Using the $\psi_j'$, we can  define an $G$-invariant Borel measurable function $\Psi'$ from $(end_G(H)\backslash Z_2')^{(2)}$ in the same way as before. 
Since $\widehat i(end_F(G))$ is a $G$-invariant null set, $end_G(H)$ is $H$-invariant, and since we are assuming $G$ is of divergence type, $end_G(H)$ is either null or conull. If it is conull, we conclude that $\Psi'$ has to be constant almost everywhere. This leads to a contradiction as before.
Since $\widehat i(end_F(H))$ is contained in $end_G(H)\cup \widehat i(end_F(G))$, it is also a null set. 
\end{proof}

\begin{lem}\label{iHborel}
$Z=\widehat i(end_F(G)\cup end_F(H))$ is a Borel set.
\end{lem}
\begin{proof}
It is well known that the image of a Borel set in a standard measure space by a countable-to-one continuous map is also Borel. See \cite[Theorem 4.12.4]{Sri} for example.
We already know $end_F(G)$ is Borel. Since $end_F(H)$ is also a Borel set, which can be shown by the argument as in the first half of Lemma \ref{endnull}, we see that $Z$ is a Borel set.
\end{proof}

Thus we conclude that $Z$ is a Borel null set.
We have proved $(1)$ of Theorem \ref{thm:main}, and
$(2)$ follows from the argument by Sullivan and Tukia in \cite{Sul82, Tuk}.
Here we shall present a simplified version of the proof of \cite[Theorem 3C]{Tuk}.

First we need to restrict the underlying $\sigma$-algebra of the measure spaces $(S^2_\infty, \mu_F)$ and $((S^2_\infty)^{(2)}, \mu_F)$ to the Borel $\sigma$-algebra.
Let us  denote $\widehat j\widehat i^{-1}$ by $\widehat k$.
Then $\widehat k\vert_{\Lambda_G\setminus Z}$ is Borel measurable and also Borel {\it directly measurable} in the sense that it sends every Borel subset of $\Lambda_G\setminus Z$ to a Borel set.
We define the pull-back measure $\widehat k^*\mu_H$ on $\Lambda_G$ by setting $(\widehat k^*\mu_H)(B):=\mu_H(\widehat k(B\backslash Z))$ for any Borel set $B$ in $\Lambda_G$.
Next, we apply the Lebesgue decomposition theorem to the measure $\widehat k^*\mu_H$ on $(\Lambda_G, \mu_G)$. 
Then we get a conull Borel subset $A$ of $\Lambda_G\setminus Z$ such that $\widehat k^*\mu_H$ absolutely continuous with respect to $\mu_G$ on $A$.
 If $(\widehat k^*\mu_H)(A)=0$, then second case in $(2)$ of Theorem \ref{thm:main} holds and we are done.
 Therefore, we assume $(\widehat k^*\mu_H)(A)> 0$ from now on. 
 Then we have the Radon-Nikodym derivative $D=d(\widehat k^*\mu_H)/d\mu_G$ on $A$, which cannot be a $\mu_G$-almost everywhere zero function.

We define a $G$-invariant ergodic measure $\nu_G$ on $\Lambda_G^{(2)}$ as follows:
$$d\nu_G(\xi_1, \xi_2)=\frac{d\mu_G(\xi_1)d\mu_G(\xi_2)}{\vert \xi_1-\xi_2\vert^{\delta_G}}.$$
We can define an $H$-invariant measure $\nu_H$ on $\Lambda_H^{(2)}$ in the same way although $\nu_H$ may not be ergodic. 
Let $\nu$ be a measure on $A^{(2)}$ defined by setting $\nu(B')=\nu_H((\widehat k\times \widehat k)(B'\cap A^{(2)})$ for any Borel set $B'$ in $\Lambda_G^{(2)}$.
Then $\nu$ can be written as follows:
$$d\nu(\xi_1, \xi_2)=\frac{d(\widehat k^*\mu_H)(\xi_1)d(\widehat k^*\mu_H)(\xi_2)}{\vert \widehat k(\xi_1)-\widehat k(\xi_2)\vert^{\delta_H}}$$
for $(\xi_1,\xi_2)\in A^{(2)}$.
Since $\widehat k$ is equivariant, $\nu$ is also $G$-invariant. Thus $d\nu/d\nu_G$ is a $G$-invariant  measurable function defined $\nu_G$-almost everywhere. Due to the ergodicity of $\nu_G$, it has to be a constant $\nu_G$-almost everywhere. We can write $d\nu/d\nu_G$ as follows:
$$\frac{d\nu}{d\nu_G}(\xi_1, \xi_2)=D(\xi_1)D(\xi_2)\frac{\vert \xi_1-\xi_2\vert^{\delta_G}}{\vert \widehat k(\xi_1)-\widehat k(\xi_2)\vert^{\delta_H}}.$$
Thus $d\nu/d\nu_G$ is a nonzero constant function $\nu_G$-almost everywhere and
this implies that for almost every quadruple $(\xi_1,\xi_2,\xi_3,\xi_4)$ in $\Lambda_G^4$, 
$$cr(\xi_1,\xi_2,\xi_3,\xi_4)^{\delta_G}=cr(\widehat k(\xi_1), \widehat k(\xi_2), \widehat k(\xi_3), \widehat k(\xi_4))^{\delta_H}.$$
Now since $G$ is non-elementary,  we have  $\delta_G>0$ by \cite[Proposition 3.8]{CS}.
Therefore,  \cite[Theorem 2B]{Tuk} implies $\delta_G=\delta_H$.
Thus the above equality shows the invariance of cross ratio by $\widehat k$, and hence that $\widehat k$ is the restriction of a M\"{o}bius transformation.

\end{document}